\documentclass[11pt]{article}

\usepackage{amsmath}
\usepackage{amssymb}
\usepackage{amscd}
\usepackage{amsthm}
\usepackage{indentfirst}

\newcommand{\N}{\ensuremath{\mathbb{N}}}
\newcommand{\Z}{\ensuremath{\mathbb{Z}}}
\newcommand{\Q}{\ensuremath{\mathbb{Q}}}

\usepackage[margin=2.8cm]{geometry}

\theoremstyle{plain}
\newtheorem{thm}{Theorem}[section]
\newtheorem{lem}[thm]{Lemma}
\newtheorem{cor}[thm]{Corollary}

\newtheorem{question}[thm]{Question}
\newtheorem{conj}[thm]{Conjecture}
\newtheorem*{fact}{Fact}

\DeclareMathOperator{\Cl}{Cl}
\DeclareMathOperator{\Gal}{Gal}
\DeclareMathOperator{\Disc}{Disc}

\DeclareMathOperator{\rank}{rk}
\DeclareMathOperator{\ord}{ord}
\DeclareMathOperator{\divisor}{div}

\DeclareMathOperator{\Jac}{Jac}
\DeclareMathOperator{\Sel}{Sel}

\newcommand{\PP}{\mathbb{P}}
\newcommand{\tors}{\text{\rm tors}}

\begin{document}

\title{A geometric approach to large class groups: a survey}

\author{Jean Gillibert \and Aaron Levin \thanks{The second author was supported in part by NSF grant DMS-1102563.}}

\date{November 2019}

\maketitle

\begin{abstract}
The purpose of this note is twofold. First, we survey results from \cite{levin07}, \cite{gl12} and \cite{bg18} on the construction of large class groups of number fields by specialization of finite covers of curves. Then we give examples of applications of these techniques.
\end{abstract}


\section{The survey}

The geometric techniques we shall report on are in fact explanations of geometric nature of a strategy which has been used from the beginning of the subject. We hope to convince the reader that this geometric viewpoint has many advantages. In particular, it clarifies the general strategy, and it allows one to obtain quantitative results.
Furthermore, it raises new questions concerning torsion subgroups of Jacobians of curves defined over number fields.

Let us point out that, for simplicity, we focus here on geometric techniques related to covers of curves. Similar results hold for covers of arbitrary varieties, see \cite{levin07} and \cite{gl12}, at the price of greater technicalities. The advantage of considering arbitrary varieties is that it provides a general framework to explain all constructions from previous authors, without exception.


\subsection{Large class groups: the folklore conjecture}

If $M$ is a finite abelian group, and if $m>1$ is an integer, we define the $m$-rank of $M$ to be the maximal integer $r$ such that $(\Z/m\Z)^r$ is a subgroup of $M$; we denote it by $\rank_m M$. If $k$ is a number field, we let $\Cl(k)$ denote the ideal class group of $k$, and $\Disc(k)$ denote the (absolute) discriminant of $k$.

The following conjecture is widely believed to be true.

\begin{conj}
\label{conj1}
Let $d>1$ and $m>1$ be two integers. Then $\rank_m\Cl(k)$ is unbounded when $k$ runs through the number fields of degree $[k:\Q]=d$. 
\end{conj}

When $m=d$, and more generally when $m$ divides $d$, this conjecture follows easily from Class Field Theory. On the other hand, when $m$ and $d$ are coprime, there is not a single case where Conjecture~\ref{conj1} is known to hold.

When constructing families of degree $d$ fields $k$ with a given lower bound on $\rank_m\Cl(k)$,  it is natural to count the number of fields constructed, ordered by discriminant. This is the quantitative aspect of Conjecture~\ref{conj1}.

For a detailed account of qualitative results towards Conjecture~\ref{conj1}, and a discussion of quantitative results, see \S{}\ref{tablesofresults}.
 

\subsection{A toy example}

In order to give a flavor of our technique, we revisit a classical construction.

\begin{fact}
Let $m\geq 3$ be an odd integer. For infinitely many odd $x\in\N$, the imaginary quadratic field $k=\Q(\sqrt{1-x^m})$ satisfies $\rank_m\Cl(k)\geq 1$.
\end{fact}

\subsubsection*{Classical proof}

If $y=\sqrt{1-x^m}$, then $y^2=1-x^m$ and $x^m=(1-y)(1+y)$. The ideal $(1-y,1+y)$ divides $2$, but is also coprime to $2$ because $x$ is odd. Thus $1-y$ and $1+y$ are coprime, and their product is an $m$-th power; hence each of them generates an ideal that is the $m$-th power of some ideal of $\mathcal{O}_k$. So there exists an ideal $\mathfrak{a}$ of $\mathcal{O}_k$ such that $(1-y)=\mathfrak{a}^m$. Therefore, the class of $\mathfrak{a}$ in $\Cl(k)$ has order dividing $m$.

The next (and hardest) step is to prove that, if some additional condition is satisfied, then the class of $\mathfrak{a}$ in $\Cl(k)$ has exact order $m$. Let us assume that the class of $\mathfrak{a}$ has order $q<m$, i.e. there exists $\alpha\in \mathcal{O}_k$ such that $\mathfrak{a}^q=(\alpha)$. One can write $m=q\ell$ for some odd $\ell>1$. It follows that $(\alpha)^\ell=(1-y)$, and hence there exists a unit $\varepsilon$ of $k$ such that $\varepsilon \alpha^\ell=1-y$. Assuming that the square-free part of $1-x^m$ is strictly smaller than $-3$, the units of the imaginary quadratic field $k$ are $\{\pm 1\}$, and hence $\alpha^\ell=1-y$ up to sign change. The result ultimately relies on the inexistence of solutions to this Diophantine equation, which one achieves by putting additional conditions on $x$. For example, Ram Murty \cite{Mur98} has shown that, if the square factor of $1-x^m$ is less than $x^{m/4}/2\sqrt{2}$, then this equation has no solution. The same result is proved in \cite{CHKP} under the condition that $x$ is a prime number $\geq 5$. Eventually, as was pointed out by Cohn \cite{Cohn03}, it follows from the work of Nagell  \cite{Nag55} that, for any odd number $x\geq 5$, one has $\rank_m\Cl(k)\geq 1$.

Quantitative results are usually obtained by \emph{ad hoc} techniques of analytic number theory, depending on the required condition on $x$.

\subsubsection*{Geometric proof}

Let $C$ be the smooth, projective, geometrically irreducible hyperelliptic curve over $\Q$ defined by the affine equation
$$
y^2=1-x^m.
$$

Let $T\in C(\Q)$ be the point with affine coordinates $(0,1)$. The integer $m$ being odd, the curve $C$ has a unique point at infinity, that we denote by $\infty$. The divisor of the rational function $1-y$ is given by
$$
\divisor(1-y)=mT-m\infty,
$$
which proves that the class of the divisor $T-\infty$ defines a rational point of order dividing $m$ in the Jacobian of $C$. It is not very hard to check that in fact this divisor class has exact order $m$ in the Jacobian of $C$.

A brief reminder on ramification in Kummer extensions : let $K$ be a local field with valuation $v$, let $m>1$ be an integer, and let $\gamma\in K^{\times}$. If  the Kummer extension $K(\sqrt[m]{\gamma})/K$ is unramified at $v$, then $v(\gamma)\equiv 0\pmod{m}$. Conversely, if $v(\gamma)\equiv 0\pmod{m}$ and the residue characteristic of $K$ is coprime to $m$, then $K(\sqrt[m]{\gamma})/K$ is unramified at $v$.

The valuation of $1-y$ at each place of $C$ is a multiple of $m$; hence the function field extension $\Q(C)(\sqrt[m]{1-y})/\Q(C)$ is unramified at each place of $C$. Therefore, this extension corresponds to an {\'e}tale cover of $C$, that we denote by $f:\tilde{C}\to C$. This is a geometrically connected cover of degree $m$, because the class of $T-\infty$ has order $m$ in the Jacobian of $C$.

Let us consider a point $P\in C(\overline{\Q})$ satisfying the following properties:
\begin{enumerate}
\item[$(i)$] for each finite place $v$ of $\Q(P)$, $v(1-y(P))\equiv 0\pmod{m}$;
\item[$(ii)$] $[\Q(\sqrt[m]{1-y(P)}):\Q(P)]=m$;
\item[$(iii)$] $\Q(P)$ is linearly disjoint from the $m$-th cyclotomic field $\Q(\mu_m)$.
\end{enumerate}

Then we claim that $\rank_m \Cl(\Q(P))\geq 1-\rank_{\Z} \mathcal{O}_{\Q(P)}^{\times}$.
In order to prove this, let us define
\begin{equation}
\label{Sm}
\Sel^m(\Q(P)):=\{\gamma\in\Q(P)^\times/(\Q(P)^\times)^m; \text{$\forall v$ finite place of $\Q(P)$, $v(\gamma)\equiv 0\pmod{m}$}\} 
\end{equation}
which is an analogue of the Selmer group for the multiplicative group over $\Q(P)$. Then we have an exact sequence
\begin{equation*}
\begin{CD}
1 @>>> \mathcal{O}_{\Q(P)}^{\times}/\big(\mathcal{O}_{\Q(P)}^{\times}
\big)^m @>>> \Sel^m(\Q(P)) @>>> \Cl(\Q(P))[m] @>>> 0. \\
\end{CD}
\end{equation*}
By condition $(i)$, the element $1-y(P)$ defines a class in $\Sel^m(\Q(P))$, which has exact order $m$ by condition $(ii)$. It follows from condition $(iii)$ that $\rank_m \mathcal{O}_{\Q(P)}^{\times}/(\mathcal{O}_{\Q(P)}^{\times})^m=\rank_{\Z} \mathcal{O}_{\Q(P)}^{\times}$. Therefore, by considering $m$-ranks in the exact sequence above, one obtains the result.

We shall now prove the existence of infinitely many points $P\in C(\overline{\Q})$ satisfying $(i)$, $(ii)$ and $(iii)$, and such that $\Q(P)$ is an imaginary quadratic field. It follows from Dirichlet's unit theorem that $\rank_m \Cl(\Q(P))\geq 1$ for such fields. On the other hand, if $\Q(P)$ is a real quadratic field, then this machinery does not yield any result, because we don't have a way to ensure that $1-y(P)$ is not a unit modulo $m$-th powers.

It follows from an appropriate version of the Chevalley-Weil theorem that, if $p$ is a prime of good reduction of $C$, then for any  point $P\in C(\overline{\Q})$, the extension $\Q(\sqrt[m]{1-y(P)})/\Q(P)$ is unramified at all places $v$ dividing $p$. For such $v$, the condition $v(1-y(P))\equiv 0\pmod{m}$ is satisfied, according to the Kummer criterion.

An immediate application of the Jacobian criterion of smoothness shows that the primes of bad reduction of $C$ are the primes dividing $2m$. We shall now deal with local conditions at such primes $p$. Let $P_0$ be the rational point of $C$ with affine coordinates $(1,0)$. Then $P_0$ is a ramification point of $x$, and $y(P_0)=0$. Let $P\in C(\overline{\Q})$ be a point which is $p$-adically close enough to $P_0$, by which we mean that, for each place $v$ of $\Q(P)$ dividing $p$, the point $P$ is close to $P_0$ for the $v$-adic topology on $C(\Q(P)_v)$. Then by elementary considerations $y(P)$ is $p$-adically close to $y(P_0)$, hence to $0$. Therefore, if $P$ is $p$-adically close enough to $P_0$, then $v(1-y(P))=0$ for each place $v$ of $\Q(P)$ dividing $p$.

Let $\Delta$ be the product of bad primes, and let $\phi:C\to\PP^1$ be the rational map defined by $\phi=\frac{x-1}{\Delta^N}$ for some integer $N$ large enough. The map $\phi$ being totally ramified at $P_0$, one can see that, for all $t\in\N$ and all bad primes $p$, the point $P_t:=\phi^{-1}(t)$ is $p$-adically close enough to the point $P_0$. Then the discussion above proves that all points $P_t$ with $t\in\N$ satisfy condition $(i)$.  In fact, it was shown in the classical proof that the map $\phi=\frac{x-1}{2}$ does the job, so we shall use that one instead.

Applying Hilbert's irreducibility theorem to the composite cover of degree $2m$
$$
\begin{CD}
\tilde{C} @>f>> C @>\phi >> \PP^1 \\
\end{CD}
$$
we obtain the existence of infinitely many $t\in\N$ such that $[\Q(f^{-1}(P_t)):\Q]=2m$. For such $t$, the field $\Q(P_t)$ is quadratic, and $\Q(f^{-1}(P_t))=\Q(\sqrt[m]{1-y(P_t)})$ is an extension of degree $m$ of $\Q(P_t)$. Hence condition $(ii)$ is satisfied. Moreover, $\Q(P_t)=\Q(\sqrt{1-(2t+1)^m})$ is imaginary quadratic, hence $(iii)$ holds (unless $3\mid m$ and $\Q(P)=\Q(\sqrt{-3})$, which we exclude). This concludes the proof of the statement.

Finally, it follows from a quantitative version of Hilbert's irreducibility theorem, due to Dvornicich and Zannier \cite{Zan}, that, given $X>0$, there exist $\gg X^{\frac{1}{m}}/\log X$ imaginary quadratic fields $\Q(P_t)$ with discriminant $|\Disc(\Q(P_t))|<X$ such that condition $(ii)$ is satisfied. This yields a quantitative version of the result.

\subsubsection*{Comments}

At first glance, the geometric proof seems more technical than the classical one. Let us list some advantages of this technique over the classical one.

A first advantage of geometry is to avoid the use of \emph{ad hoc} ``tricks''. 
More precisely, in the classical method one uses two tricks: the first one is to ensure that $1-y$ is the $m$-th power of some ideal of $k$, which is done by requiring congruence conditions on the variables. In the geometric case, this relies on the Chevalley-Weil theorem and some additional condition, namely: there exists a rational point on $C$ at which the map $C\to\PP^1$ is totally ramified. The second trick, which is the hardest, is to prove that, for every $\ell>1$ dividing $m$, the equation $\alpha^\ell=1-y$ has no solution $\alpha\in k$. In the geometric world, this follows immediately from Hilbert's irreducibility theorem, without any additional technicality.

Another nice feature of the geometric approach: the use of Hilbert's irreducibility theorem automatically gives us a quantitative version of the result. This should be compared to the specific analytic number theory machinery that has been used previously on these quantitative class group problems.

A final advantage comes from the (arguable) fact that it is relatively easier to build {\'e}tale covers of curves than everywhere unramified extensions of number fields.


\subsection{General specialization results}

Following the lines of the geometric proof of the ``toy example'' above, one obtains the following general statement.

\begin{thm}
\label{thold}
Let~$C$ be a smooth, projective, geometrically irreducible curve over $\Q$, let $\Jac(C)$ be the Jacobian of~$C$, and let ${m>1}$ be an integer. Assume that~$C$ admits a finite morphism $C\to\PP^1$ of degree~$d$, totally ramified over some point belonging to $\PP^1(\Q)$.
Then there exist infinitely many (isomorphism classes of) number fields~$k$ with $[k:\Q]=d$  such that
\begin{equation}
\label{eold}
\rank_m \Cl(k) \geq \rank_m \Jac(C)(\Q)_{\tors} - \rank_{\Z} \mathcal{O}_{k}^\times.
\end{equation}
\end{thm}

Inspired by the technique introduced in \cite{levin07}, this theorem was proved in~\cite{gl12} in the case when $C$ is a superelliptic curve defined by a ``nice equation'' (see Corollary~3.1 of \cite{gl12}). The version above is proved in \cite{bg18}, the base field being $\Q$ for simplicity. In \S{}\ref{Aaron} we state a variant of this result, in which the assumption that the morphism $C\to\PP^1$ is totally ramified over some rational point is replaced by a more technical one.

While  Theorem~\ref{thold} is quite general, its applicability in concrete cases is impaired by the presence of the negative term 
${- \rank_{\Z} \mathcal{O}_{k}^\times}$ on the right: the rank of the unit group of the field $k$ tends to be large, especially if $d>2$.

This deficiency is avoided in the following theorem, which constitutes the main result of \cite{bg18}.
Let us denote by $\rank_{\mu_m}\Jac(C)$ the maximal integer~$r$ such that $\Jac(C)$ has a $\Gal(\overline{\Q}/\Q)$-submodule isomorphic to $\mu_m^r$.

\begin{thm}
\label{thnew}
In the set-up of Theorem~\ref{thold},
there exist infinitely many number fields~$k$ with ${[k:\Q]=d}$ such that
\begin{equation}
\label{enew}
\rank_m \Cl(k) \geq \rank_{\mu_m} \Jac(C).
\end{equation}
\end{thm}

In contradistinction with the proof of Theorem~\ref{thold} which is based on Kummer theory, the proof of Theorem~\ref{thnew} relies on Class Field Theory. It can be seen as a generalization of the constructions of Mestre~\cite{Mes1, Mes3, Mes2, Mes92}.

In both Theorems~\ref{thold} and~\ref{thnew} the ``infinitely many'' can be made quantitative, the fields being ordered by discriminant.

\begin{thm}
\label{thquant}
Let ${\phi\in \Q(C)}$ be the rational function defining the morphism ${C\to \PP^1}$ appearing in both Theorems~\ref{thold} and~\ref{thnew}. 
Assume that there exists a rational function ${x\in \Q(C)}$ of degree~$n$ such that ${\Q(C)=\Q(\phi,x)}$. Then, for sufficiently large positive~$X$,
in both these theorems  the number of isomorphism classes of the fields~$k$ satisfying~\eqref{eold} or~\eqref{enew}, respectively, and such that ${|\Disc(k)| \le X}$ is $\gg X^{1/2n(d-1)}/\log X$.
\end{thm}

In a recent work \cite{BL16}, Bilu and Luca improved the quantitative version of Hilbert's irreducibility theorem given by Dvornicich and Zannier, on which our quantitative results are based. We underline the fact that any improvement of quantitative HIT automatically yields a similar improvement of our quantitative results.

In the case when $C$ is a hyperelliptic curve with a rational Weierstrass point, it is possible to improve slightly the quantitative result. More precisely, we obtain in \cite{gl12} the following quantitative version of Theorem~\ref{thold} for such curves.

\begin{cor}
\label{HypCor}
Let $C$ be a smooth projective hyperelliptic curve over $\Q$ with a rational Weierstrass point, and let $m>1$ be an integer.  Let $g$ denote the genus of $C$.  Then there exist $\gg X^{\frac{1}{2g+1}}/\log X$ imaginary (resp. real) quadratic number fields $k$ with ${|\Disc(k)|<X}$ and
\begin{align*}
& \rank_m \Cl(k)\geq \rank_m \Jac(C)(\Q)_{\tors} \\
\text{(resp. } &\rank_m \Cl(k)\geq \rank_m \Jac(C)(\Q)_{\tors}-1\text{).}
\end{align*}
\end{cor}

In view of the statement above, the following question arises immediately:

\begin{question}
\label{Q1}
Let $m>1$ be an integer.  Do there exist hyperelliptic curves $C$ over $\Q$ with $\rank_m \Jac(C)(\Q)_{\tors}$ arbitrarily large?
\end{question}

According to Corollary~\ref{HypCor}, a positive answer to this question would provide a proof of Conjecture~\ref{conj1} in the $d=2$ case, provided the curves have rational Weierstrass points.

For $m=2$, the question above has a positive answer. Apart from this easy case, very little is known. To our knowledge, the best general result is the following: given $m>1$, there exist hyperelliptic curves $C$ over $\Q$ with $\rank_m \Jac(C)(\Q)_{\tors}\geq 2$. This allows one to derive Yamamoto's result from Corollary~\ref{HypCor} (see \S{}\ref{Yamamoto}).

Finally, it follows from Theorem~\ref{thnew} that Corollary~\ref{HypCor} and Question~\ref{Q1} have natural analogues in which $\rank_m \Jac(C)(\Q)_{\tors}$ is replaced by $\rank_{\mu_m} \Jac(C)$. Unfortunately, we have not been able to find examples of hyperelliptic curves over $\Q$ with large $\rank_{\mu_m} \Jac(C)$.


\subsection{Record of known results towards Conjecture~\ref{conj1}}
\label{tablesofresults}

In this section we give a brief summary of the history of results on the problem of finding infinite families of number fields of degree $d$ over $\mathbb{Q}$ with ideal class groups of large $m$-rank (see Tables 1 and 2 below for a more comprehensive list of results).  The earliest such result could be considered to be Gauss' result determining, in modern terms, the $2$-rank of the class group of a quadratic number field in terms of the primes dividing the discriminant of the quadratic field.  In particular, it follows from Gauss' result that the $2$-rank of the ideal class group of a quadratic number field can be made arbitrarily large.  In contrast to Gauss' result, there is not a single quadratic number field $k$ and prime $p\neq 2$ for which it is known that $\rank_p \Cl(k)>6$, although the Cohen-Lenstra heuristics \cite{Len, Len2} predict that for any given positive integer $r$, a positive proportion of quadratic fields $k$ should have $\rank_p \Cl(k)=r$.

\begin{table}[ht]
\label{table1}
\caption{Values of $m$ and $r$ for which it is known that there exist infinitely many quadratic fields $k$ with $\rank_m\Cl(k)\geq r$ (we let $r=\infty$ if $\rank_m\Cl(k)$ can be made arbitrarily large).
All results in this table can be recovered by applying Corollary~\ref{HypCor}, except Mestre's ones, which are applications of Theorem~\ref{thnew}.}
\bigskip
\centering
\begin{tabular}{|l|c|c|c|c|}
\hline
Author(s) & Year & Type & $m$ & $r$\\
\hline
Gauss & 19th c. & imaginary, real & $2$ & $\infty$\\
\hline
Nagell \cite{Nag, Nag2}& 1922 & imaginary & $>1$ & $1$ \\
\hline
Yamamoto \cite{Yam} & 1970 & imaginary & $>1$& $2$\\
\hline
Yamamoto \cite{Yam}, Weinberger \cite{Wei} &1970, 1973 & real & $>1$& $1$\\
\hline
Craig \cite{Cra73} & 1973 & imaginary & $3$& $3$\\
&  & real & $3$& $2$\\
\hline
Craig \cite{Cra77} & 1977 & imaginary & $3$& $4$\\
 & & real & $3$& $3$\\
\hline
Diaz y Diaz \cite{Dia} & 1978 & real & $3$& $4$\\
\hline
Mestre \cite{Mes1, Mes3, Mes2} & 1980 & imaginary, real & $5,7$& $2$\\
\hline
Mestre \cite{Mes92} & 1992 & imaginary, real & $5$& $3$\\
\hline
\end{tabular}
\bigskip

\end{table}

\begin{table}[ht]
\label{table2}
\caption{Values of $m$, $d$, and $r$ for which it is known that there exist infinitely many number fields $k$ of degree $d$ with $\rank_m\Cl(k)\geq r$.
All results in this table can be recovered by applying variants of Theorem~\ref{thold}, except the cases when $m=2$, which follow from variants of Theorem~\ref{thnew}.}
\bigskip
\centering
\begin{tabular}{|l|c|c|c|c|}
\hline
Author(s) & Year & $m$ & $d$ & $r$\\
\hline
Brumer, Rosen \cite{Bru, Ros} & 1965 & $>1$ & $d=m$ & $\infty$\\
\hline
Uchida \cite{Uch} & 1974 & $>1$ & $3$ & $1$\\
\hline
Ishida \cite{Ish} & 1975 & $2$ & prime & $d-1$\\
\hline
Azuhata, Ichimura \cite{Ich} & 1984 & $>1$ & $>1$ & $\left\lfloor \frac{d}{2}\right\rfloor$\\
\hline
Nakano \cite{Nak5, Nak2} & 1984 & $>1$ & $>1$ & $\left\lfloor \frac{d}{2}\right\rfloor+1$\\
 & 1985 & $2$ & $>1$ & $d$\\
\hline
Nakano \cite{Nak4} & 1988 & $2$ & $3$ & $6$\\
\hline
Levin \cite{levin07} & 2007 & $>1$ & $>1$ & $\left\lceil \left\lfloor\frac{d+1}{2}\right\rfloor+\frac{d}{m-1}-m\right\rceil$\\
\hline
Kulkarni \cite{Kul} & 2017 & $2$ & $3$ & $8$\\
\hline
\end{tabular}
\end{table}

The first constructive result on $m$-ranks of class groups for arbitrary $m$ was given in 1922 by Nagell \cite{Nag, Nag2}, who proved that for any positive integer $m$, there exist infinitely many imaginary quadratic number fields whose class group has an element of order $m$ (in particular, there are infinitely many imaginary quadratic fields with class number divisible by $m$).  Nagell's result has since been reproved by a number of different authors (e.g., \cite{Cho}, \cite{Hum}, \cite{Kur}).  Nearly fifty years later, working independently, Yamamoto \cite{Yam} and Weinberger \cite{Wei} extended Nagell's result to real quadratic fields.  Soon after, Uchida \cite{Uch} proved the analogous result for cubic cyclic fields.  In 1984, Azuhata and Ichimura \cite{Ich} succeeded in extending Nagell's result to number fields of arbitrary degree.  In fact, they proved that for any integers $m,d>1$ and any nonnegative integers $r_1$, $r_2$, with $r_1+2r_2=d$, there exist infinitely many number fields $k$ of degree $d=[k:\mathbb{Q}]$ with $r_1$ real places and $r_2$ complex places such that
\begin{equation}
\label{Ich}
\rank_m \Cl(k)\geq r_2.
\end{equation}
The right-hand side of (\ref{Ich}) was subsequently improved to $r_2+1$ by Nakano \cite{Nak5, Nak2}.  Choosing $r_2$ as large as possible, we thus obtain, for any $m$, infinitely many number fields $k$ of degree $d>1$ with
\begin{equation}
\label{eqN1}
\rank_m \Cl(k)\geq \left\lfloor\frac{d}{2}\right\rfloor+1,
\end{equation}
where $\lfloor \cdot \rfloor$ and $\lceil \cdot \rceil$ denote the greatest and least integer functions, respectively.  For general $m$ and $d$, \eqref{eqN1} is the best result that is known on producing number fields of degree $d$ with a class group of large $m$-rank.  In \cite{levin07} it was shown that there exist infinitely many number fields $k$ of degree $d$ satisfying $\rank_m \Cl(k)\geq \left\lceil \left\lfloor\frac{d+1}{2}\right\rfloor+\frac{d}{m-1}-m\right\rceil$, improving \eqref{eqN1} when $d\geq m^2$.

For certain special values of $m$ and $d$, slightly more is known.  Of particular note to us are Mestre's papers \cite{Mes1,Mes3,Mes2,Mes92} giving the best known results for $m=5,7$ and $d=2$. Mestre's method can be seen as an application of Theorem~\ref{thnew} (see \S{}\ref{Mestre} below).

Recently, progress has been made on obtaining quantitative results on counting the number fields in the above results.  Murty \cite{Mur2} gave the first results in this direction, obtaining quantitative versions of the theorems of Nagell and Yamamoto-Weinberger.  His results have since been improved by, among others, Soundararajan \cite{Sou} in the imaginary quadratic case and Yu \cite{Yu} in the real quadratic case.  In higher degrees, Hern{\'a}ndez and Luca \cite{Luca} gave the first such result for cubic number fields, while Bilu and Luca \cite{Bilu} succeeded in proving a quantitative theorem for number fields of arbitrary degree.  Bilu and Luca's result was improved in \cite{levin07}, where a quantitative version of Azuhata and Ichimura's result was given. In \S{}\ref{Aaron}, we show how it is possible to derive from Theorem~\ref{thold} a short proof of this result.


\section{The examples}


This section is devoted to examples of applications of Theorems \ref{thold} and \ref{thnew}. Each of these examples is obtained by revisiting previous constructions. In certain cases, this yields new quantitative results.


\subsection{Yamamoto's result}
\label{Yamamoto}

In \cite{Yam}, Yamamoto proved that, for any integer $m>1$, there exist infinitely many imaginary (resp. real) quadratic fields $k$ with $\rank_m\Cl(k)\geq 2$ (resp. $\rank_m\Cl(k)\geq 1$).

In order to recover this result via geometry, we proved the following in \cite{gl12}.

\begin{lem}
\label{hlem}
Let $\lambda\in \Q^{\times}$, $\lambda\neq \pm 1$, and let $m>1$ be an integer.  Let $C$ be the smooth projective hyperelliptic curve defined over $\Q$ by the affine equation $y^2=x^{2m}-(1+\lambda^2)x^m+\lambda^2$. Then $C$ has a rational Weierstrass point, and $\rank_m \Jac(C)(\Q)_{\tors}\geq 2$.
\end{lem}

Applying Corollary~\ref{HypCor} to this situation, we obtained the following quantitative version of Yamamoto's result \cite[Corollary~3.4]{gl12}.

\begin{cor}
\label{ri}
Let $m>1$ be an integer.  There exist $\gg X^{\frac{1}{2m-1}}/\log X$ imaginary (resp. real) quadratic number fields $k$ with ${|\Disc(k)|<X}$ and $\rank_m \Cl(k)\geq 2$ (resp. $\rank_m \Cl(k)\geq 1$).
\end{cor}

If $m$ is odd, then Byeon \cite{Bye} and Yu \cite{Yu} have proved, for imaginary and real quadratic fields, respectively, the better lower bound of $\gg X^{1/m-\epsilon}$.  If $m$ is even, in the real quadratic case a lower bound of $\gg X^{1/m}$ was proved by Chakraborty, Luca, and Mukhopadhyay \cite{CLM}.  The imaginary quadratic case of Corollary~\ref{ri} with $m$ even appears to be a new result of \cite{gl12}.


\subsection{$3$-ranks of quadratic fields: a construction of Craig}
\label{Craig}

In \cite{Cra77}, Craig constructed infinitely many imaginary (resp.\ real) quadratic fields $k$ with $\rank_3\Cl(k)\geq 4$ (resp. with $\rank_3\Cl(k)\geq 3$).  We prove quantitative versions of Craig's result and show how his constructions yield a hyperelliptic curve whose Jacobian has a rational subgroup isomorphic to $(\Z/3\Z)^4$.

Let $f$ be the polynomial
$$
f(x,y,z)=x^{6}+y^{6}+z^{6}-2x^3y^3-2x^3z^3-2y^3z^3.
$$
The idea in \cite{Cra77} is to find a nontrivial parametric family of solutions to the equations
\begin{equation*}
f(x_0,y_0,z_0)=f(x_1,y_1,z_1)=f(x_2,y_2,z_2).
\end{equation*}
Since 
\begin{equation*}
f(x,y,z)=(x^3+y^3-z^3)^2-4x^3y^3=(x^3-y^3+z^3)^2-4x^3z^3=(-x^3+y^3+z^3)^2-4y^3z^3,
\end{equation*}
it suffices to find solutions to
\begin{align}
x_1z_1=x_0z_0, \quad x_2y_2=x_0y_0,\label{Cr4}\\
x_1^3-y_1^3+z_1^3=-(x_0^3-y_0^3+z_0^3),\label{Cr5}\\
x_2^3+y_2^3-z_2^3=-(x_0^3+y_0^3-z_0^3).\label{Cr6}
\end{align}
Craig gives a two-parameter family of solutions to \eqref{Cr4}, \eqref{Cr5}, and \eqref{Cr6} in terms of $\alpha$, $\beta$, and $\gamma$ satisfying $\alpha+\beta+\gamma=0$.  We refer the reader to \cite{Cra77} for the rather involved formulas.  We specialize Craig's solution by setting $\alpha=0$, $\beta=t$, and $\gamma=-t$.  This gives a polynomial $h(t)=f(x_0(t),y_0(t),z_0(t))$ of degree $141$.  Let $C$ be the (nonsingular projective model of the) hyperelliptic curve defined by $Y^2=h(t)$.  We have the four identities (where $x_0=x_0(t)$, $y_0=y_0(t)$, etc.),
\begin{align*}
\left(Y+(x_0^3+y_0^3-z_0^3)\right)\left(Y-(x_0^3+y_0^3-z_0^3)\right)=-4x_0^3y_0^3,\\
\left(Y+(x_0^3-y_0^3+z_0^3)\right)\left(Y-(x_0^3-y_0^3+z_0^3)\right)=-4x_0^3z_0^3,\\
\left(Y+(x_1^3+y_1^3-z_1^3)\right)\left(Y-(x_1^3+y_1^3-z_1^3)\right)=-4x_1^3y_1^3,\\
\left(Y+(-x_2^3+y_2^3+z_2^3)\right)\left(Y-(-x_2^3+y_2^3+z_2^3)\right)=-4y_2^3z_2^3.
\end{align*}
It follows that there are divisors $D_1$, $D_2$, $D_3$, and $D_4$ on $C$ such that 
\begin{align*}
(Y+x_0^3+y_0^3-z_0^3)=3D_1,\\
(Y+x_0^3-y_0^3+z_0^3)=3D_2,\\
(Y+x_1^3+y_1^3-z_1^3)=3D_3,\\
(Y-x_2^3+y_2^3+z_2^3)=3D_4.
\end{align*}
Using Magma, it is easy to verify that $D_1$, $D_2$, $D_3$, and $D_4$ give independent $3$-torsion elements of $\Jac(C)(\mathbb{Q})$ (to simplify calculations, this can be done modulo $p=7$, a prime of good reduction of $C$).  Thus, we arrive at the following result.

\begin{thm}
Let $C$ be the  hyperelliptic curve defined by $Y^2=h(t)$.  Then
$$
\rank_3 \Jac(C)(\mathbb{Q})_{\tors}\geq 4.
$$
\end{thm}

Since $h$ has odd degree, $C$ has a rational Weierstrass point, and so Corollary~\ref{HypCor} applies.

\begin{cor}
There exist $\gg X^{\frac{1}{141}}/\log X$ imaginary (resp. real) quadratic fields $k$ with ${|\Disc(k)|<X}$ and $\rank_3 \Cl(k)\geq 4$ (resp. $\rank_3 \Cl(k)\geq 3$).
\end{cor}


\subsection{$5$-ranks of quadratic fields: a construction of Mestre}
\label{Mestre}

In \cite{Mes92}, Mestre proved the existence of infinitely many imaginary and real quadratic fields $k$ with $\rank_5\Cl(k)\geq 3$. We briefly review his construction. For the reader's convenience, we stick to the original notation. Mestre constructs:
\begin{enumerate}
\item[(1)] a genus $5$ hyperelliptic curve $C$ defined over $\Q$, which admits three rational Weierstrass points;
\item[(2)] three elliptic curves $E_1$, $E_2$ and $E_3$ defined over $\Q$, each of them endowed with an isogeny $\varphi_i:E_i\to F_i$ with kernel $\Z/5\Z$;
\item[(3)] three independent Galois covers $\tau_i:C\to F_i$ with group $(\Z/2\Z)^2$.
\end{enumerate}

The existence of the maps $\tau_i$ implies that the Jacobian of~$C$ splits, and that each of the $F_i$ is an isogenus factor of $\Jac(C)$ via an isogeny of degree $4$. More precisely, there exists an abelian surface $B$ and an isogeny
$$
F_1\times F_2\times F_3 \times B \longrightarrow \Jac(C)
$$
whose degree is a power of $2$.

On the other hand, the dual isogeny $\hat{\varphi_i}:F_i\to E_i$ has kernel $\mu_5$, because the kernel of the $\hat{\varphi_i}$ is the Cartier dual of the kernel of  $\varphi_i$. Hence $\Jac(C)$ contains $\mu_5^3$ as a subgroup, which means in our terminology that $\rank_{\mu_5} \Jac(C) \geq 3$.

Applying Theorem~\ref{thnew} to this situation, we obtain the following quantitative version of Mestre's result.

\begin{thm}
There exist $\gg X^{\frac{1}{11}}/\log X$ imaginary (resp. real) quadratic fields $k$ with ${|\Disc(k)|<X}$ such that $\rank_5 \Cl(k) \geq 3$.
\end{thm}


\subsection{Higher degree fields}
\label{Aaron}

Let us fix integers $m,r>1$ with $(r,m)=1$. Consider a superelliptic curve $C$ defined by an affine equation of the form
$$
y^m=a_0\prod_{i=1}^r(x-a_i),
$$
where $a_1,\ldots, a_r$ are pairwise distinct rational numbers, and $a_0\in\Q^\times$. Then $x$ and $y$ are rational functions on $C$ with $\deg x=m$ and $\deg y=r$. Since $(r,m)=1$, the curve $C$ has a unique point at infinity, that we denote by $\infty$, and $x$ and $y$ are totally ramified at that point. For each $i$, let $P_i$ be the rational point on $C$ with affine coordinates $(a_i,0)$. Then one has
$$
\divisor(x-a_i)=mP_i-m\infty.
$$

A classical argument shows that the divisor classes $(P_i-\infty)_{i=1}^r$ generate a subgroup of $\Jac(C)(\Q)$ isomorphic to $(\Z/m\Z)^{r-1}$.

Applying Theorem~\ref{thold} to the map $x:C\to \PP^1$, one recovers the result of Brumer and Rosen (first line in Table~2). By considering the map $y:C\to \PP^1$, one recovers results of Azuhata and Ichimura (line~4 in Table~2). Using Hilbert's irreducibility theorem, quantitative versions of these results were obtained in \cite{levin07}.

Using other maps, it was shown in \cite{levin07} that in some situations it is possible to improve on Nakano's inequality \eqref{eqN1} (line~5 in Table~2).

\begin{thm}
\label{thLev2}
Let $m,d>1$ be integers with $d>(m-1)^2$.  There exist $\gg X^{\frac{1}{(m+1)d-1}}/\log X$ number fields $k$ of degree $d$ with ${|\Disc(k)|<X}$ and
\begin{equation*}
\rank_m \Cl(k)\geq \left\lceil \left\lfloor\frac{d+1}{2}\right\rfloor+\frac{d}{m-1}-m\right\rceil.
\end{equation*}
\end{thm}

A detailed proof of this theorem is given in \cite{levin07}, but it is possible to give a simpler proof by using the following variant of Theorem~\ref{thold}.

\begin{thm}
\label{variant1}
Let $C$ be a smooth projective geometrically irreducible curve over $\Q$, let $\Jac(C)$ be the Jacobian of $C$, and let $m>1$ be an integer. Let $s=\rank_m \Jac(C)(\Q)_{\tors}$, and let $D_1,\dots, D_s$ be divisors on $C$ whose classes in $\Jac(C)(\Q)$ generate a subgroup isomorphic to $(\Z/m\Z)^s$. Let $g_1,\dots,g_s$ be rational functions on $C$ such that $\divisor(g_i)=mD_i$ for all $i$. Assume that there exists a finite map $\phi:C\to\PP^1$ of degree $d$ such that, for all $t\in\N$, the point $P_t:=\phi^{-1}(t)$ has the property that
\begin{equation}
\label{SelCond}
g_1(P_t),\dots,g_s(P_t) ~\text{define classes in}~ \Sel^m(\Q(P_t)),
\end{equation}
where $\Sel^m$ is defined in \eqref{Sm}. Then there exist infinitely many $t\in\N$ such that $[\Q(P_t):\Q]=d$ and
$$
\rank_m \Cl(\Q(P_t))\geq s-\rank_\Z \mathcal{O}_{\Q(P_t)}^\times.
$$
Moreover, there are infinitely many isomorphism classes of such fields $\Q(P_t)$.
\end{thm}

This statement is a generalization of Theorem~\ref{thold}, in which the condition on the existence of a totally ramified point for the map $\phi$ is replaced by a more technical one. In explicit examples, this technical condition usually comes from a congruence condition on the coordinates of the point $P_t$ (see proof of Theorem~\ref{thLev2} below).

Theorem~\ref{variant1} can be proved along the lines of the toy example. We refer to the proof of \cite[Theorem~2.4]{gl12}, in which all needed arguments already appear. Needless to say, Theorem~\ref{variant1} admits the same quantitative version as the previous ones, as stated in Theorem~\ref{thquant}.

\begin{proof}[Proof of Theorem~\ref{thLev2}]
Let $m,d>1$ be integers with $d>(m-1)^2$.  Let $r$ be the largest integer such that $r-\left\lfloor\frac{r}{m}\right\rfloor\leq d$ and $(r,m)=1$.  It is easily checked that $r\geq d+\frac{d}{m-1}-m+1$.  Let $C$ be the curve defined by
$$
y^m=h(x)=-(x-a_1^m)\prod_{i=2}^r (x+a_i^m)
$$
where $a_1,\ldots, a_r$ are certain carefully chosen integers \cite[Lemma~3.1]{levin07}. For $i=2,\dots,r$, we let $g_i:=x+a_i^m$. Then $\divisor(g_i)=mP_i-m\infty$ where $P_i=(-a_i^m,0)$, and, as noted above, the divisor classes $(P_i-\infty)_{i=2}^r$ generate a subgroup isomorphic to $(\Z/m\Z)^{r-1}$ in $\Jac(C)(\Q)$.

Let $f(x)$ be the Taylor series for $\sqrt[m]{h(x)}$ at $x=0$ truncated to degree $\left\lfloor\frac{r}{m}\right\rfloor-1$ with $f(0)=\prod_{i=1}^ra_i$. Then $f$ is defined over $\Q$, and
$$
\ord_x(f^m-h)\geq \left\lfloor\frac{r}{m}\right\rfloor \geq r-d.
$$

Let $b$ be the lowest common denominator of the coefficients of $f$.  Let $\psi:C\to\PP^1$ be the rational function defined by
$$
\psi:=\frac{b(y-f)}{x^{r-d}}.
$$
Then one computes \cite[Lemma~3.5]{levin07} that $\psi$ has degree $d$.

Let $\Delta_0$ be the product of prime numbers dividing the discriminant of $h$. Having chosen $a_1,\ldots, a_r$ properly, it can be shown \cite[Lemma~3.3]{levin07} that there exists an integer $c_0$ such that, for each integer $c\equiv c_0 \pmod{\Delta_0}$, the point $Q_c:=\psi^{-1}(c)$ has the property that $g_2(Q_c),\dots,g_r(Q_c)$ define classes in $\Sel^m(\Q(Q_c))$.

If we define $\phi:C\to\PP^1$ by
$$
\phi:=\frac{\psi-c_0}{\Delta_0},
$$
then $\phi$ also has degree $d$ and, for all $t\in\N$, the point $P_t:=\phi^{-1}(t)$ satisfies condition \eqref{SelCond} from Theorem~\ref{variant1} with respect to the functions $g_2,\dots,g_r$.

Finally, it follows from \cite[Lemma~3.4]{levin07} that 
$$
\Disc(\mathbb{Q}(P_t))=O(t^{(m+1)d-1})
$$
and $\mathbb{Q}(P_t)$ has at most two real places for $t\gg 0$. The result follows from Theorem~\ref{variant1}.
\end{proof}


\subsection*{Acknowledgements}

The first author warmly thanks Yuri Bilu for very inspiring conversations, and for his feedback on a preliminary version of this note.



\providecommand{\bysame}{\leavevmode\hbox to3em{\hrulefill}\thinspace}


\bigskip

\textsc{Jean Gillibert}, Institut de Math{\'e}matiques de Toulouse, CNRS UMR 5219, 118 route de Narbonne, 31062 Toulouse Cedex 9, France.

\emph{E-mail address:} \texttt{jean.gillibert@math.univ-toulouse.fr}
\medskip

\textsc{Aaron Levin}, Department of Mathematics, Michigan State University, 619 Red Cedar Road, East Lansing, MI 48824.

\emph{E-mail address:} \texttt{adlevin@math.msu.edu}


\end{document}